\documentclass[11pt,reqno]{amsart}
\usepackage{amssymb}

\usepackage{euscript}

\newtheorem{Lemma}{Lemma}[section]
\newtheorem{Theorem}[Lemma]{Theorem}
\newtheorem{Proposition}[Lemma]{Proposition}
\newtheorem{Corollary}[Lemma]{Corollary}
\newtheorem{Claim}[Lemma]{Claim}

\newcommand{\R}{\mathbb R}
\newcommand{\Z}{\mathbb Z}
\newcommand{\supp}{\operatorname{supp}}
\newcommand{\de}[2]{\frac{\partial #1}{\partial #2}}
\newcommand{\p}{\partial}
\newcommand{\graph}{\operatorname{graph}}

\begin{document}
\title[The X-ray transform]{The X-ray transform on a general family of curves on Finsler surfaces}

\author[Yernat M. Assylbekov]{Yernat M. Assylbekov}
\address{Department of Mathematics, University of Washington, Seattle, WA 98195-4350, USA}
\email{y\_assylbekov@yahoo.com}

\author[Nurlan S. Dairbekov]{Nurlan S. Dairbekov}
\address{Kazakh British Technical University,
Tole bi 59, 050000 Almaty, Kazakhstan }
\email{Nurlan.Dairbekov@gmail.com}

\begin{abstract}
For a compact oriented Finsler surface with smooth boundary,
we consider the scalar and vector integral geometry problems over a general family of curves
running between boundary points and parametrized by arclength.
We impose a natural condition which results in the no conjugate points
condition in the case when the curves in question are geodesic lines.
Our main theorem generalizes Mukhometov's theorem in several directions.

We also consider these problems on a closed oriented Finsler surface.
In this case the integral geometry problems make sense provided that
sufficiently many curves in the family are periodic.
To this end, we assume that the induced flow on the unit circle bundle of the surface is Anosov.
Also, we study the cohomological equation of thermostats without conjugate points.
\end{abstract}

\maketitle


\section{Introduction}
\subsection{Surfaces with boundary}
Let $M$ be a  compact oriented surface (two-dimensional manifold) with boundary. 
Let $\Gamma$ be a family of regular parametrized curves in $M$ between boundary points such that
\begin{enumerate}
\item[(A1)] For every interior point $x \in M$ and every $v\in
T_xM\setminus\{0\}$, there is exactly one curve
in $\Gamma$ passing through $x$ in the direction of $v$
(considering the curves obtained by shift of the parameter to be the same curve).
\end{enumerate}
We denote by $\gamma_{x,v}$ such a curve with its parameter shifted so
that $\gamma_{x,v}(0)=x$, $\dot{\gamma}_{x,v}(0)=cv$ ($c>0$).
We suppose our family $\Gamma$ to be smooth in the sense that
\begin{enumerate}
\item[(A2)] The partial map
$$
(x,v,t)\mapsto \gamma_{x,v}(t)
$$                    
is $C^{\infty}$-smooth.
\end{enumerate}

We assume our family $\Gamma$ to have no conjugate points in the sense that
\begin{enumerate}
\item[(A3)]  The partial map $\exp^\Gamma_x:T_xM\setminus\{0\}\to M$, 
defined as
$$
\exp^\Gamma_x(t\dot\gamma_{x,v}(0))=\gamma_{x,v}(t),\quad t>0,
$$
is a local diffeomorphism for every $x\in M$.
\end{enumerate}

Let $f$ be a smooth function and $\alpha$ be a smooth 1-form on $M$.
Define the {\it X-ray transform} of the function $\psi(x,v)=f(x)+\alpha_x(v)$ by
$$
I_{\Gamma}\psi(\gamma)=\int_{\gamma}\{f(\gamma(t))+\alpha_{\gamma(t)}(\dot{\gamma}(t))\}\,dt,\quad
\gamma\in\Gamma,
$$
meaning the integral over $\gamma$ to stand for the integral over the
domain of the parameter.
This transform embraces the scalar ($\alpha=0$) and vector ($f=0$) X-ray transforms as particular
instances, and cannot be decoupled in general, as can be seen, for
example, from the case of $\Gamma$ being the family of magnetic geodesics (see \cite{DPSU}).
It is clear that $I_\Gamma$ has a non-trivial kernel since
$$
I_{\Gamma}dh(\gamma)=\int_{\gamma}dh=0
$$
for all $\gamma\in\Gamma$ and for any $h\in C^\infty(M)$ such that $h|_{\p M}=0$.
A natural question of integral geometry is whether these are the only elements of the kernel.

Since the integral of a scalar function over a curve respects the parametrization of the curve,
we assume $M$ to be furnished with some Finsler metric $F$ and consider the condition:
\begin{enumerate}
\item[(A4)] All curves in $\Gamma$ are parametrized by arclength with respect to $F$:
$$
F(\dot\gamma(t))=1
$$
for all $t$ and every $\gamma\in \Gamma$.
\end{enumerate}

Our first result is the following:
\begin{Theorem}\label{theorem A}
Let $M$ be a compact oriented surface with boundary, let $F$ be a Finsler metric on $M$, and let
$\Gamma$ be a family of curves in $M$ between boundary points satisfying conditions (A1)-(A4).
Suppose $\psi(x,v)=f(x)+\alpha_x(v)$, where $f$ is a smooth function and
$\alpha$ is a smooth $1$-form on $M$.
Then $I_\Gamma\psi(\gamma)=0$ for all $\gamma\in \Gamma$ if and only
if $f=0$ and $\alpha=dh$ for some $h\in C^\infty(M)$ such that $h|_{\p{M}}=0$.
\end{Theorem}

In \cite{Mukh2}, R. G. Mukhometov proved a similar result for a general family of curves on
subdomains of the Euclidean plane. In higher dimension, the scalar integral geometry problem for a
real analytic family of curves was solved by B. Frigyik, P. Stefanov and G. Uhlmann \cite{FSU}. It
is worth noting that the literature on the question is abandoned in the case when $\Gamma$ is the
family of geodesics of a Riemannian metric.

In a purely vectorial case (when $f=0$), we may freely reparametrize curves in $\Gamma$ without
influencing the X-ray transform.
Therefore, we have the following consequence of Theorem \ref{theorem A}:

\begin{Corollary}\label{theorem B}
Let $M$ be a compact oriented surface with boundary, let
$\Gamma$ be a family of curves in $M$ between boundary points satisfying (A1)--(A3),
and let $\alpha$ be a smooth $1$-form on $M$.
Then $I_\Gamma\alpha(\gamma)=0$ for all $\gamma\in \Gamma$ if and only
if $\alpha=dh$ for some $h\in C^\infty(M)$ such that $h|_{\p{M}}=0$.
\end{Corollary}

In \cite{HS}, S. Holman and P. Stefanov solved the vector integral geometry problem
for a real analytic family of curves in any dimension.

\subsection{Closed surfaces}
A similar problem for closed surfaces is interesting as well.
Let $M$ be a  closed (i.e., compact and boundaryless) oriented surface, and let
$\Gamma$ be a family of regular parametrized curves
each of which is defined on the whole real axis.

Assume (A1) and (A2),
and let $f$ be a smooth function and $\alpha$ a smooth 1-form on $M$.
Define the {\it X-ray transform} of the function $\psi(x,v)=f(x)+\alpha_x(v)$ by
$$
I_{\Gamma}\psi(\gamma)
=\int_{\gamma}\{f(\gamma(t))+\alpha_{\gamma(t)}(\dot{\gamma}(t))\}\,dt
\quad\text{for periodic }\gamma\in\Gamma,
$$
meaning the integral over $\gamma$ to stand for the integral over the least period.

As before, $I_\Gamma$ has a non-trivial kernel since
$I_{\Gamma}dh(\gamma)=0$
for all periodic $\gamma\in\Gamma$ and for any $h\in C^\infty(M)$.
In this setting, the integral geometry problem asks whether these are the only elements in the
kernel. Of course, sufficiently many curves in $\Gamma$ must be periodic for this to be true. If
$F$ is a Finsler metric on $M$ and (A4) holds, $\Gamma$ defines a flow $\phi_t$ on the unit circle
bundle by the rule
$$
\phi_t(x,v)\mapsto (\gamma_{x,v}(t),\dot\gamma_{x,v}(t)).
$$
Instead of (A3) we require
\begin{enumerate}
\item[(A3')] The flow $\phi_t$ on $SM$ is Anosov.
\end{enumerate}

Recall that the Anosov property means that there is a continuous invariant splitting
$T(SM)=\mathbb{R} \mathbf F\oplus E^{u}\oplus E^{s}$ ($\mathbf F$ being the generator of the flow)
in such a way that
there are constants $C>0$ and $0<\rho<1<\eta$ such that
for all $t>0$ we have
\[\|d\phi_{-t}|_{E^{u}}\|\leq C\,\eta^{-t}\;\;\;\;\mbox{\rm
and}\;\;\;\|d\phi_{t}|_{E^{s}}\|\leq C\,\rho^{t},\]
where the norms are taken with respect to the Sasaki type Riemannian metric on $SM$ induced by the
Finsler metric $F$.

\begin{Theorem}\label{theorem C}
Let $(M,F)$ be a closed oriented Finsler surface, and let
$\Gamma$ be a family of curves in $M$ satisfying (A1)-(A2), (A3'), and (A4).
Suppose $\psi(x,v)=f(x)+\alpha_x(v)$, where $f$ is a smooth function and
$\alpha$ is a smooth $1$-form on $M$.
Then $I_\Gamma\psi(\gamma)=0$ for all $\gamma\in \Gamma$ if and only
if $f=0$ and $\alpha=dh$ for some $h\in C^\infty(M)$.
\end{Theorem}

In \cite{GK}, V. Guillemin and D. Kazhdan proved Theorem \ref{theorem C}
for $F$ a negatively curved Riemannian metric and $\Gamma$
the family of unit-speed geodesics of $F$.
A similar result was obtained by G. P. Paternain for magnetic flows  in \cite{Pkam}.
All these results were based on Fourier analysis.
In \cite{DPkam}, N. S. Dairbekov and G. P. Paternain proved Theorem \ref{theorem C}
for the case of the magnetic flow on a Riemannian surface.
The same result was proved in \cite{DPentr} by
N. S. Dairbekov and G. P. Paternain for thermostats on Riemannian surfaces
and in \cite{DPrigid} for magnetic flows on Finsler manifolds of any dimension.

\subsection{Thermostats}
The above-mentioned general families of curves on surfaces are conveniently defined in terms of
(generalized) thermostats.

Consider as before a compact oriented surface $M$ and a Finsler metric $F$ on $M$. Let $SM$ be the unit
circle bundle of $(M,F)$ and $\pi:SM\to M$ be the canonical projection, $\pi(x,v)=x$.
For any $\lambda\in C^\infty(TM\setminus\{0\})$, consider the equation
\begin{equation}\label{termo}
\frac{D \dot{\gamma}}{dt}=\lambda(\gamma, \dot{\gamma})i \dot{\gamma},
\end{equation}
where $i$ indicates the rotation by $\pi/2$ according to the orientation of $M$.
Every solution of \eqref{termo} has constant speed,
and we restrict ourselves to unit-speed solutions. In this case, it suffices to assume
that $\lambda\in C^\infty(SM)$.
A curve parametrized by arclenth and satisfying \eqref{termo} will be referred to
as {\it $\lambda$-geodesic}.
We call the triple $(M,F,\lambda)$
a (generalized) {\it thermostat}. In case $\lambda$ is (the pullback of) a function on $M$,
we have a magnetic system (see, for example, \cite{DPSU}).
If $\lambda$ is a $1$-form (regarded as a function on $SM$),
we have a Gaussian thermostat (see, for example, \cite{DPentr}).

In the case when $M$ has boundary, we assume that the thermostat in question is {\it nontrapping}
in the sense that every $\lambda$-geodesic has finite exit times both in the positive and negative
directions. If $M$ is a closed surface, we assume that that every $\lambda$-geodesic is complete,
i.e., defined on the whole real axis. In these cases, we can declare $\Gamma$ to be the family of
$\lambda$-geodesics of our thermostat. Then (A1), (A2), and (A4) are obviously satisfied. On the
other hand, it is easy to see that the converse is true as well.

\begin{Theorem}\label{Gamma-defines-flow}
If $\Gamma$ is a family of curves on a compact oriented Finsler surface $(M,F)$, satisfying (A1),
(A2) and (A4), then it is the family of $\lambda$-geodesics for a suitable $\lambda$.
\end{Theorem}
\begin{proof} Define $\lambda$ as
$$
\lambda(x,v,t)
=\Big\langle\frac{D\dot{\gamma}_{x,v}(t)}{dt},i\dot{\gamma}_{x,v}(t)\Big\rangle_{\dot{\gamma}_{x,v}(t)},
$$
the inner product on the right-hand side taken with respect to the fundamental tensor in Finsler
geometry:
$$
g_{ij}(x,v)=\frac12[F^2]_{v^iv^j}(x,v).
$$

Condition (A1) implies that the function $\lambda$ does not depend on $t$.
Then $\Gamma$ becomes the family of $\lambda$-geodesics of the thermostat $(M,F,\lambda)$.
\end{proof}

\subsection{Cohomological equation}
The cohomological (or kinetic) equation of a flow is simply $\mathbf F(u)=\psi$, where
$\mathbf F$ is the infinitesimal generator of the flow  and $u, \psi$ are functions on $SM$.
The importance of the cohomological equation in dynamical systems is well known; it arises
for example in the study of invariant measures, conjugacy problems, reparametrizations,
rigidity questions and inverse problems.

It follows from \cite{Ghy} that under condition (A3') the flow $\phi_t$ is transitive and, by
the smooth {Liv\v cic} theorem \cite{LMM}, Theorem \ref{theorem C} admits an equivalent
restatement as an inverse problem for the cohomological equation. On taking Theorem
\ref{Gamma-defines-flow} into account, we formulate the corresponding result in terms of
thermostats.

\begin{Theorem}\label{theorem C'}
Let $(M,F,\lambda)$ be an Anosov thermostat on a closed oriented Finsler surface $M$, and let
${\bf F}$ be the infinitesimal generator of the thermostat flow $\phi_t$.
Suppose $\psi(x,v)=f(x)+\alpha_x(v)$, where $f$ is a smooth function and
$\alpha$ is a smooth $1$-form on $M$.
Then the cohomological equation ${\bf F}(u)=\psi$
has a solution $u\in C^{\infty}(SM)$ if and only if $f=0$ and the form $\alpha$ is exact.
\end{Theorem}

In case $\psi$ is a scalar function on $M$ (i.e., $\alpha=0$), this theorem can be generalized to
thermostats without conjugate points. We say that a thermostat has no conjugate points if the
family of $\lambda$-geodesics satisfies condition (A3), i.e., if the exponential map
\begin{equation}\label{exponential}
\exp^\lambda_x(v):=\gamma_{x,v}(|v|), \quad v\ne 0,
\end{equation}
is a local diffeomorphism for every $x$. Here $\gamma_{x,v}(t)$ is a unit-speed $\lambda$-geodesic
with $\gamma_{x,v}(0)=x$ and $\dot\gamma_{x,v}(0)=v/|v|$.

\begin{Theorem}\label{theorem E}
Let $(M,F,\lambda)$ be a thermostat without conjugate points on a closed oriented
Finsler surface $M$.
Suppose $\psi(x,v)=f(x)$, where $f$ is a smooth function on $M$.
Then the cohomological equation ${\bf F}(u)=\psi$
has a solution $u\in C^{\infty}(SM)$ if and only if $f=0$.
\end{Theorem}

Note that if a flow is Anosov, then there are no conjugate points (see Subsection~\ref{4.1}).
There are several interesting examples of geodesic flows without conjugate
points which are not Anosov and have regions of positive curvature (see \cite{BBB,DPcohomag}).
Examples of magnetic flows without conjugate points which are not Anosov
are also given in \cite{DPcohomag}.

Theorem \ref{theorem A}, too, is proved by a reduction to an inverse problem
for the cohomological equation
on a thermostat. The reduction follows the arguments in \cite{DInt} and is based on
the previous observations by V.~A.~Sharafutdinov in \cite{Shar1}.

In all the cases, the cohomological equation is analyzed by means of Pestov type identities that
we derive in each case.

\subsection{Structure of the paper} The organization of the paper is as follows.
In Section~2 we combine certain preliminaries concerning thermostats on Finsler surfaces, as
well as derive differential and integral Pestov type identities.
Section~3 is devoted to the proof of Theorem~\ref{theorem A}.
Section~4 contains the proof of Theorem~\ref{theorem C}. Here we also give
sufficient conditions in terms of the Finsler metric and thermostat data for a thermostat flow
to be Anosov
(see Theorem~\ref{theorem D}). The closing Section~5 contains the proof of Theorem~\ref{theorem E}.


\section{Pestov identity}
\subsection{Canonical coframing}\label{cancof}
By \cite[Chapter 4]{BCS} for a given Finsler surface $(M,F)$ it is possible to define a canonical coframing $(\omega_1, \omega_2, \omega_3)$ on $SM$ that satisfies the following structural equations:
\begin{align}
& d\omega_1=-\omega_2 \wedge \omega_3,\label{dw1}\\
& d\omega_2=-\omega_3 \wedge (\omega_1-I\omega_2),\label{dw2}\\
& d\omega_3=-(K\omega_1 - J\omega_3) \wedge \omega_2.\label{dw3}
\end{align}
where $I$, $K$ and $J$ are smooth functions on $SM$. The function $I$ is called the main scalar of the structure. When the Finsler structure is Riemannian, $K$ is the Gaussian curvature.

Consider the vector fields $(X,H,V)$ dual to $(\omega_1,\omega_2,\omega_3)$. As a consequence of (\ref{dw1}--\ref{dw3}) they satisfy the commutation relations
\begin{equation}\label{cmr1}
\begin{aligned}
&[V,X]=H,\\
&[H,V]=X+IH+JV,\\
&[X,H]=KV.
\end{aligned}
\end{equation}

Let $\lambda$ be a smooth function on $SM$ and let ${\bf F}=X+\lambda V$
be the generating vector field of the thermostat $(M,F,\lambda)$. From \eqref{cmr1} we obtain:
\begin{equation}\label{cmr2}
\begin{aligned}
& [V,{\bf F}]=H+V(\lambda)V,\\
& [H,V]={\bf F}+IH+(J-\lambda)V,\\
& [{\bf F},H]=\{K-H(\lambda)-\lambda J+\lambda^2\}V-\lambda {\bf F}-\lambda IH.
\end{aligned}
\end{equation}

\subsection{Pestov identity}
\begin{Lemma}[Pestov identity]\label{pestov}For every smooth function $u:SM \to \mathbb{R}$ we have
\begin{align*}
2Hu \cdot V{\bf F}u = ({\bf F}u)^2+(Hu)^2-(K-H(\lambda)-\lambda J+\lambda^2)(Vu)^2\\
+{\bf F}(Hu \cdot Vu)-H(Vu \cdot {\bf F}u)+V(Hu \cdot {\bf F}u)\\
+{\bf F}u \cdot (IHu+JVu)+Hu \cdot Vu \cdot (\lambda I+V(\lambda)).
\end{align*}
\end{Lemma}

\begin{proof}
Using the commutation formulas, we deduce:
\begin{align*}
2Hu \cdot V{\bf F}u-V&(Hu \cdot {\bf F}u)=Hu \cdot V{\bf F}u-VHu \cdot {\bf F}u\\
&=Hu \cdot({\bf F}Vu+[V,{\bf F}]u)-{\bf F}u \cdot(HVu+[V,H]u)\\
&=Hu \cdot({\bf F}Vu+Hu+V(\lambda)Vu)\\
&\quad-{\bf F}u \cdot(HVu-{\bf F}u-IHu-(J-\lambda)Vu)\\
&=(Hu)^2+({\bf F}u)^2+Hu \cdot {\bf F}Vu-HVu \cdot {\bf F}u+I{\bf F}u \cdot Hu\\
&\quad+(J-\lambda){\bf F}u \cdot Vu+V(\lambda)Hu \cdot Vu\\
&=(Hu)^2+({\bf F}u)^2+{\bf F}(Hu\cdot Vu)-H(Vu\cdot {\bf F}u)-[{\bf F},H]u\cdot Vu\\
&\quad+I{\bf F}u\cdot Hu+(J-\lambda){\bf F}u\cdot Vu+V(\lambda)Hu\cdot Vu\\
&=(Hu)^2+({\bf F}u)^2+{\bf F}(Hu\cdot Vu)-H(Vu\cdot {\bf F}u)\\
&\quad+(-(K-H(\lambda)-\lambda J+\lambda^2)Vu+\lambda {\bf F}u+\lambda IHu)\cdot Vu\\
&\quad+I{\bf F}u\cdot Hu+(J-\lambda){\bf F}u\cdot Vu+V(\lambda)Hu\cdot Vu\\
&=(Hu)^2+({\bf F}u)^2+{\bf F}(Hu\cdot Vu)-H(Vu\cdot {\bf F}u)\\
&\quad-(K-H(\lambda)-\lambda J+\lambda^2)(Vu)^2+I{\bf F}u\cdot Hu+J{\bf F}u\cdot Vu\\
&\quad+(\lambda I+V(\lambda))Hu\cdot Vu
\end{align*}
which is equivalent to the Pestov identity.
\end{proof}

For $G$ a vector field and $\Theta$ a differential form, Cartan's formula for
the Lie derivative reads:
$$
\mathcal L_{G}\Theta=d(i_{G}\Theta)+i_{G}d\Theta.
$$

Now let $\Theta:=\omega_1\wedge\omega_2\wedge\omega_3$. This volume form gives rise to the
Liouville measure $d\mu$ of $SM$.

\begin{Lemma}\label{lie}
We have:
\begin{align}
&\mathcal L_{{\bf F}}\Theta=(\lambda I+V(\lambda))\Theta, \label{lie1}\\
&\mathcal L_{H}\Theta=-J\Theta, \label{lie2}\\
&\mathcal L_{V}\Theta=I\Theta. \label{lie3}
\end{align}
\end{Lemma}

\begin{proof}Using equations (\ref{dw1}--\ref{dw3})
$$
\mathcal L_{X}\Theta
=d(i_{X}\Theta)=d(\omega_2\wedge\omega_3)=d\omega_2\wedge\omega_3-\omega_2\wedge d\omega_3=0.
$$
Since ${\bf F}=X+\lambda V$, we get
\[\mathcal L_{{\bf F}}\Theta=\mathcal L_{X}\Theta+\mathcal L_{\lambda V}\Theta=d(i_{\lambda V}\Theta)=-d(\lambda \omega_2\wedge\omega_1)=(\lambda I+V(\lambda))\Theta.\]
Similarly, $\mathcal L_{H}\Theta=-J\Theta$, $\mathcal L_{V}\Theta=I\Theta$.
\end{proof}

\subsection{Pestov integral identity}
Integrate the Pestov identity over $SM$ against the Liouville measure $d\mu$
by making use of the Stokes Theorem and (\ref{lie1}--\ref{lie3}) to get
\begin{align*}
\int_{SM}2Hu \cdot V{\bf F}u\,d\mu&=\int_{SM}({\bf F}u)^2\,d\mu+\int_{SM}(Hu)^2\,d\mu\\
&-\int_{SM}\{K-H(\lambda)-\lambda J+\lambda^2\}(Vu)^2\,d\mu\\
&+\int_{\p(SM)}\{(Hu\cdot Vu)i_{\bf F}\Theta+({\bf F}u\cdot Hu)i_V\Theta-({\bf F}u\cdot Vu)i_H\Theta\}.
\end{align*}

Since $i_V \Theta=\omega_1\wedge\omega_2$ vanishes when restricted to $\p(SM)$, we have
$$
\int_{\p(SM)}({\bf F}u\cdot Hu)i_V\Theta=0.
$$
So we get
\begin{equation}
\begin{aligned}
\int_{SM}2Hu \cdot V{\bf F}u\,d\mu&=\int_{SM}({\bf F}u)^2\,d\mu+\int_{SM}(Hu)^2\,d\mu\\
&-\int_{SM}\{K-H(\lambda)-\lambda J+\lambda^2\}(Vu)^2\,d\mu\\
&+\int_{\p(SM)}\{(Hu\cdot Vu)i_{\bf F}\Theta-({\bf F}u\cdot Vu)i_H\Theta\}.
\end{aligned}\label{int-id1}
\end{equation}

By commutation relations, we have
\[{\bf F}Vu=V{\bf F}u-Hu-V(\lambda)Vu.\]
Therefore,
\begin{align*}
({\bf F}Vu)^2&=(V{\bf F}u)^2+(Hu)^2+(V(\lambda)Vu)^2-2V{\bf F}u\cdot Hu+2Hu\cdot V(\lambda)Vu\\
&-2V{\bf F}u\cdot V(\lambda)Vu\\
&=(V{\bf F}u)^2+(Hu)^2+(V(\lambda)Vu)^2-2V{\bf F}u\cdot Hu+2Hu\cdot V(\lambda)Vu\\
&-2V{\bf F}u\cdot V(\lambda)Vu+2V(\lambda){\bf F}Vu\cdot Vu-2V(\lambda){\bf F}Vu\cdot Vu\\
&=(V{\bf F}u)^2+(Hu)^2+(V(\lambda)Vu)^2-2V{\bf F}u\cdot Hu+2Hu\cdot V(\lambda)Vu\\
&-2V(\lambda)[V,{\bf F}] Vu-2V(\lambda){\bf F}Vu\cdot Vu\\
&=(V{\bf F}u)^2+(Hu)^2-(V(\lambda)Vu)^2-2V{\bf F}u\cdot Hu-2V(\lambda){\bf F}Vu\cdot Vu.
\end{align*}
Since
\[-2V(\lambda){\bf F}Vu\cdot Vu=-{\bf F}(V(\lambda)(Vu)^2)+(Vu)^2 {\bf F}V(\lambda)\]
we obtain:
\begin{equation}\label{dif-eq}
\begin{aligned}
({\bf F}Vu)^2&=(V{\bf F}u)^2+(Hu)^2-(V(\lambda)Vu)^2-2V{\bf F}u\cdot Hu\\
&-{\bf F}(V(\lambda)(Vu)^2)+(Vu)^2 {\bf F}V(\lambda).
\end{aligned}
\end{equation}

Integrating it over $SM$ we get
\begin{equation}\label{int-id2}
\begin{aligned}
\int_{SM}2Hu\cdot V{\bf F}u\,d\mu&
=\int_{SM}(V{\bf F}u)^2\,d\mu+\int_{SM}(Hu)^2\,d\mu-\int_{SM}({\bf F}Vu)^2\,d\mu\\
&+\int_{SM}\{\mathbf FV(\lambda)+\lambda IV(\lambda)\}(Vu)^2\,d\mu\\
&-\int_{\p(SM)}V(\lambda)(Vu)^2i_{\bf F}\Theta,
\end{aligned}
\end{equation}
since by the Stokes Theorem and (\ref{lie1})

\begin{align*}
-\int_{SM}{\bf F}(V(\lambda)(Vu)^2)\Theta
=&\int_{SM}\lambda IV(\lambda)(Vu)^2\Theta+\int_{SM}(V(\lambda))^2(Vu)^2\Theta\\
&-\int_{\p(SM)}V(\lambda)(Vu)^2i_{\bf F}\Theta.
\end{align*}

Combining (\ref{int-id1}) and (\ref{int-id2}), we come to the final integral identity:

\begin{Theorem}[Pestov integral identity]\label{1st-int-id}
\begin{multline*}
\int_{SM}({\bf F}Vu)^2\,d\mu-\int_{SM}\mathbb{K}(Vu)^2\,d\mu+\int_{\p(SM)}\omega(u)\\
=\int_{SM}(V{\bf F}u)^2\,d\mu-\int_{SM}({\bf F}u)^2\,d\mu,
\end{multline*}
where
\begin{equation}\label{def-omega}
\omega(u):=\{(Hu\cdot Vu)+V(\lambda)(Vu)^2\}i_{\bf F}\Theta-({\bf F}u\cdot Vu)i_H\Theta,
\end{equation}
and $\mathbb{K}:=K-H(\lambda)-\lambda J+\lambda^2+{\bf F}V(\lambda)+\lambda IV(\lambda)$.
\end{Theorem}
\medskip


\section{Proof of Theorem \ref{theorem A}}

As mentioned, Theorem \ref{theorem A} is proved by reduction to the cohomological equation,
followed by the analysis of the latter by means of Pestov type identities. The reduction is
performed in Subsections \ref{3.1} and \ref{3.2} and is based on observations in \cite{Shar1},
also used in \cite{DInt}, which we formulate below in Theorem~\ref{zero-bound-theorem} and Proposition
\ref{shar}. Subsection \ref{3.3} is devoted to adaptation of the Pestov integral identity (Theorem
\ref{1st-int-id}) to the case under study. In Subsection \ref{3.4} we derive one more integral
identity which we use to prove the main result. The proof of Theorem \ref{theorem A} is completed
in Subsection \ref{3.5}.

\subsection{Preparation}\label{3.1}
By Theorem \ref{Gamma-defines-flow} we can assume that \begin{equation}\label{x-ray=0}
\int_{\gamma}\psi(\gamma,\dot\gamma)\,dt=0
\end{equation}
for all $\lambda$-geodesics $\gamma$ with endpoints on $\p M$. The aim of this section is to prove
the next theorem. \begin{Theorem}\label{zero-bound-theorem}Let $(M,F,\lambda)$ be a nontrapping
thermostat without conjugate points. Suppose $\psi(x,v)$ is a smooth function on $SM$ and
\eqref{x-ray=0} holds for every $\lambda$-geodesic $\gamma$ with endpoints on $\p M$, then
\begin{equation}\label{zero-bound0}
\psi|_{S(\p M)}=0.
\end{equation}
\end{Theorem}
\begin{proof}
First of all, we extend $\psi$ to a positively homogeneous function
of degree zero on $TM\setminus\{0\}$.

Fix $x\in\partial M$ and $v \in S_x(\partial M)$. Let $n\in S_x M$ be an inward unit vector at
$x$. For $\varepsilon>0$ put $v_\varepsilon=v+\varepsilon n$ and consider the $\lambda$-geodesic
$\gamma_\varepsilon=\gamma_{x,v_\varepsilon}$. Let $\tau_\varepsilon$ is the first time at which
$\gamma_{\varepsilon}$ meets the boundary, $\gamma_{\varepsilon}(\tau_\varepsilon)\in \partial M$.
So $\gamma_\varepsilon:[0,\tau_\varepsilon]\to M$, $\gamma_\varepsilon(0)\in\partial M$,
$y_\varepsilon:=\gamma_\varepsilon(\tau_\varepsilon)\in\partial M$, and $\gamma_\varepsilon(t)\in
M^\text{int}$ for $0<t<\tau_\varepsilon$. We separately consider two possible cases. First, there
is a sequence $0<\varepsilon_k\to 0$ such that $\tau_{\varepsilon_k}\to 0$ as $k\to\infty$.
Second, there is $\tau_0>0$ such that $\tau_\varepsilon\geq \tau_0>0$ for all
$0<\varepsilon\leq\varepsilon_0$.
In the first case, assume \eqref{zero-bound0} fails; for definiteness,
$$
\psi(x,v)>0.
$$
For $k$ large enough, the points $(\gamma_{\varepsilon_k}(t),\dot{\gamma}_{\varepsilon_k}(t))$
belong to any prescribed neighbourhood of $(x,v)$ for all $t\in[0,\tau_{\varepsilon_k}]$.
Therefore, the latter inequality implies that the integrand in
$$
\int^{\tau_{\varepsilon_k}}_0 \psi(\gamma_{\varepsilon_k}(t),\dot{\gamma}_{\varepsilon_k}(t))\,dt
$$
is strictly positive on $(0,\tau_{\varepsilon_k})$. Hence, this integral is strictly positive,
which contradicts the hypothesis of the theorem.

Now, we consider the second case. Fix $\varepsilon\in(0,\varepsilon_0)$
and let $s\mapsto x_s\in \partial M$ $(0\leq s<\delta)$ be a parametrization of $\partial M$ near
$x$ such that $x_0=x$ and $\frac{dx_s}{ds}\big|_{s=0}=v$.

We first assume that $\gamma_\varepsilon$ meets $\partial M$ at
$y_\varepsilon=\gamma_\varepsilon(\tau_\varepsilon)$ transversally.
Then for $s$ small enough there is a unique
$\lambda$-geodesic $\gamma_{\varepsilon,s}$ from $x_s$ to $y_\varepsilon$ in $M$.
We choose a parametrization of $\gamma_{\varepsilon,s}$ so as to have
$\gamma_{\varepsilon,s}:[-\tau_{\varepsilon,s},0]\to M$,
$\gamma_{\varepsilon,s}(-\tau_{\varepsilon,s})=x_s$, $\gamma_{\varepsilon,s}(0)=y_{\varepsilon}$.
Moreover, $\tau_{\varepsilon,s}$ depends smoothly on $s$ and $ \gamma_{\varepsilon,s}(t)$ depends
smoothly on $(s,t)$.
Henceforth we accordingly shift a parameter on $\gamma_\varepsilon$ so that
$\gamma_\varepsilon=\gamma_{\varepsilon,0}$.

By hypothesis,
$$
\int_{-\tau_{\varepsilon,s}}^0 \psi(\gamma_{\varepsilon,s}(t),\dot\gamma_{\varepsilon,s}(t))\,dt=0.
$$
Taking the derivative with respect to $s$ at $s=0$, we get
\begin{multline}\label{derive}
\psi(x,v_\varepsilon/|v_\varepsilon|)\frac{d\tau_{\varepsilon,s}}{ds}\Big|_{s=0}
+\int_{-\tau_{\varepsilon}}^0
\left\{\frac{\partial \psi}{\partial x^k}(\gamma_\varepsilon(t),\dot\gamma_{\varepsilon}(t)) J_\varepsilon^k(t)\right.
\\+\left.\frac{\partial \psi}{\partial v^k}(\gamma_\varepsilon(t),\dot\gamma_{\varepsilon}(t)) \frac{DJ_\varepsilon^k(t)}{dt}\right\}\,dt=0,
\end{multline}
where $J_\varepsilon(t)=\frac{\partial\gamma_{\varepsilon,s}(t)}{\partial s}\big|_{s=0}$
is the variation (Jacobi) field along $\gamma_{\varepsilon}$.

Differentiating the identity $\gamma_{\varepsilon,s}(-\tau_{\varepsilon,s})=x_s$ with respect to
$s$, at $s=0$ we get
$$
J_\varepsilon(-\tau_{\varepsilon})-\dot\gamma_\varepsilon(-\tau_\varepsilon)\frac{d\tau_{\varepsilon,s}}{ds}\Big|_{s=0}=v.
$$

On putting
$$
A_\varepsilon=-\frac{d\tau_{\varepsilon,s}}{ds}\Big|_{s=0}
$$
and using the equality 
$\dot\gamma_\varepsilon(-\tau_\varepsilon)=(v+\varepsilon n)/|v+\varepsilon n|
=v+O(\varepsilon)$, we get
\begin{equation*}
J_\varepsilon(-\tau_{\varepsilon})=v-A_\varepsilon\dot\gamma_\varepsilon(-\tau_\varepsilon)
=(1-A_\varepsilon)\dot\gamma_\varepsilon(-\tau_\varepsilon)+O(\varepsilon).
\end{equation*}

Hence, $J_\varepsilon$ has the following boundary conditions:
\begin{equation}\label{J-shift}
J_\varepsilon(-\tau_{\varepsilon})=(1-A_\varepsilon)\dot\gamma_\varepsilon(-\tau_\varepsilon)+O(\varepsilon) ,
\quad J_\varepsilon(0)=0.
\end{equation}

As soon as a Jacobi field depends linearly on boundary conditions, from \eqref{J-shift} we deduce
\begin{equation}\label{jacob-eta}
J_\varepsilon(t)=-\frac{1-A_\varepsilon}{\tau_\varepsilon}t\dot\gamma_\varepsilon(t)+\tilde J_{\varepsilon}(t),
\end{equation}
where $\tilde J_{\varepsilon}$ is a Jacobi field along $\gamma_\varepsilon$ with boundary conditions
\begin{equation}\label{eta-init}
\tilde J_{\varepsilon}(-\tau_\varepsilon)=O(\varepsilon),\quad \tilde J_{\varepsilon}(0)=0.
\end{equation}
Using these conditions together with the Jacobi equation, which we do not derive here
since it can be done similarly as in \cite{DPrigid}, we conclude that
\begin{equation}\label{estimate}
\tilde J_\varepsilon=O(\varepsilon)
\end{equation}
in an appropriate $C^1$-norm.

To evaluate $A_\varepsilon$,
put $c_{\varepsilon}(s,t)=\gamma_{\varepsilon,s}((\tau_{\varepsilon,s}/\tau_\varepsilon)t)$.
Then $\tau_{\varepsilon,s}$ is the length of the curve
$c_\varepsilon(s,\cdot):[-\tau_\varepsilon,0]\to M$ and the variation field of $c_{\varepsilon}(s,t)$
is 
$$
V(t)=\frac{\partial c_{\varepsilon}(s,t)}{\partial s}\Big|_{s=0}
=\frac{\partial \gamma_{\varepsilon,s}((\tau_{\varepsilon,s}/\tau_\varepsilon)t)}{\partial s}\Big|_{s=0}
=J_\varepsilon(t)-\frac{A_\varepsilon}{\tau_\varepsilon} t\dot\gamma_\varepsilon(t)
=-\frac{t\dot\gamma_\varepsilon(t)}{\tau_\varepsilon}+O(\varepsilon).
$$ 
The first variation formula for length, together with \eqref{jacob-eta}, now gives
\begin{multline*}
\frac{d\tau_{\varepsilon,s}}{ds}\Big|_{s=0}
=-\int_{-\tau_{\varepsilon}}^0\left\langle V(t),
\frac{D\dot\gamma_{\varepsilon}}{dt}\right\rangle_{\dot\gamma_{\varepsilon}(t)}\,dt
-\langle V(-\tau_{\varepsilon}),\dot\gamma_\varepsilon(-\tau_{\varepsilon})\rangle_{\dot\gamma_{\varepsilon}(-\tau_\varepsilon)}
\\
=-\int_{-\tau_{\varepsilon}}^0\left\langle -\frac{t\dot\gamma_\varepsilon(t)}{\tau_\varepsilon},
\frac{D\dot\gamma_{\varepsilon}}{dt}\right\rangle_{\dot\gamma_{\varepsilon}(t)}\,dt
-\langle \dot\gamma_\varepsilon(-\tau_{\varepsilon}),\dot\gamma_\varepsilon(-\tau_{\varepsilon})\rangle_{\dot\gamma_{\varepsilon}(-\tau_\varepsilon)}
+O(\varepsilon)
\\
=-1+O(\varepsilon).
\end{multline*}

Hence $A_\varepsilon=1+O(\varepsilon)$. From
\eqref{jacob-eta} and \eqref{estimate} we then get:
$J_\varepsilon=O(\varepsilon)$
in an appropriate $C^1$-norm.
Therefore, using \eqref{derive} we conclude that
\begin{equation}\label{final-est}
\psi(x,v)=\psi(x,v_\varepsilon/|v_\varepsilon|)+O(\varepsilon)=O(\varepsilon).
\end{equation}

We recall that the above argument was carried out under the assumption that
$\gamma_\varepsilon$ intersects $\partial M$ transversally at the point
$y_\varepsilon=\gamma_\varepsilon(\tau_\varepsilon)$.

To get rid of this assumption, we invoke the following:

\begin{Proposition}\cite[Theorem 3.7, Ch. IX]{Saks}\label{saks} 
Given a plane set $R$, let $P$ be a subset of $R$ at every point of which the set $R$ has
an extreme tangent parallel to a fixed straight line $D$. Then the orthogonal projection of $P$
on the line at right angles to $D$ is of linear measure zero.
\end{Proposition}

Applying this proposition in polar coordinates related to the exponential map \eqref{exponential}
at $x$, we conclude that $\gamma_\varepsilon$ meets $\partial M$ transversally at $y_\varepsilon$
for almost every $\varepsilon$. Hence, \eqref{final-est} holds for all $\varepsilon$, which implies 
the claim of the theorem.
\end{proof}

\subsection{Reduction to the kinetic equation}\label{3.2}
If $\psi(x,v)=f(x)+\alpha_x(v)$ satisfies the conditions of Theorem \ref{theorem A}, we know 
from Theorem \ref{zero-bound-theorem} that $\psi(x,v)=0$ for $(x,v)\in S(\p M)$. As soon as 
$\psi(x,-v)=0$ too, we have $f(x)=0$ and $\alpha_x(v)=0$ for $x\in \partial M$ and $v\in T_x(\partial M)$.
The following obvious proposition can also be regarded as an easy consequence of \cite[Lemma~2.2]{Shar1}.

\begin{Proposition}\label{shar}
Let $g$ be a Riemannian metric on $M$ and let $n$ be the inward unit normal to $\partial M$ in $M$.
If $\theta$ is a smooth function on $\partial M$, then 
there is $h\in C^\infty(M)$ such that 
$h|_{\partial M}=0$ and $\frac{\partial h}{\partial n}\Big|_{\partial M}=\theta$.
\end{Proposition}

Considering any Riemannian metric $g$ on $M$ and taking $\theta=\alpha(n)$, 
it follows that the function $\tilde \psi=\psi-dh$ has the following property: the equality
\begin{equation}\label{zero-bound-normal}
\tilde\psi(x,v)=0
\end{equation}
holds for all $x\in \partial M$ and $v\in T_x M$.

Now we reduce the proof of Theorem \ref{theorem A} to an inverse problem for a kinetic equation. In view of \eqref{zero-bound0} and \eqref{zero-bound-normal}, we may henceforth assume that the function $\psi$ itself vanishes on the boundary:
\begin{equation}\label{zero-bound-final}
\psi|_{\p(SM)}=0.
\end{equation}

Further, without loss of generality, we assume that $M$ is a subset of a closed smooth surface
$U$. We extend $F$ to a Finsler metric on $U$ and extend $\lambda$ to a smooth function on $SU$,
thus obtaining a thermostat $(U,F,\lambda)$.
We extend $\psi$ from $SM$ to  $SU$ by zero, preserving the notation.
The boundary condition \eqref{zero-bound-final} guarantees that the so-obtained function $\psi$
is continuous on $U$ and belongs to the Sobolev space $H^1(SU)$ of
square-integrable functions with square-integrable first-order derivatives.

Since $(M,F,\lambda)$ is nontrapping, there is no complete $\lambda$-geodesic
which would be contained entirely in $M$. Therefore, for any $(x,v)\in SM$ there is a number $\tau(x,v)$ such that $\gamma_{x,v}(\tau(x,v))\notin M$. We define a function $u:SU \to {\mathbb R}$ to be
\begin{equation}\label{def-chi}
u(x,v)=\int^{\tau(x,v)}_0 \psi(\phi_t(x,v))\,dt.
\end{equation}

Note that the value of $u(x,v)$ is independent of the choice of $\tau(x,v)$. This follows from \eqref{x-ray=0} and the fact that $\psi$ vanishes on $SU\setminus SM^\text{int}$.

Call a point $(x,v)\in SM$ {\it regular} if the $\lambda$-geodesic $\gamma_{x,v}$ intersects $\p M$ transversally from either side, and if the open segment of $\gamma_{x,v}$ between the basepoint $x$ and the point of intersection lies entirely in $M^\text{int}$. We denote by $RM\subset SM$ the set of all regular points. It is clear that $RM$ is open in $SM$ and has full measure in $SM$.

\begin{Lemma}\label{chi-prop}
The function $u:SU\to \R$ has the following properties:
\begin{itemize}
\item[(i)] $u|_{S(U\setminus M)}=0$,

\item[(ii)] $u\in H^1(SU)\cap C(SU)\cap C^{\infty}(RM)$,

\item[(iii)] $u$ is $C^1$ smooth along the orbits of the thermostat flow $\phi$ and satisfies the following kinetic equation on $SU$:
\begin{equation}\label{kinetic-equation}
{\bf F}u(x,v)=-\psi(x,v).
\end{equation}
\end{itemize}
\end{Lemma}
\begin{proof}
Statement (i) is a direct consequence of \eqref{x-ray=0} and \eqref{def-chi}.

To prove (ii) take any point $(x,v)\in SU$. Since
$\gamma_{x,v}(\tau(x,v))\notin M$, we can choose a small one-dimensional subspace $\Phi$ in $U$ transversally intersecting $\gamma_{x,v}$ at the point $\gamma_{x,v}(\tau(x,v))$ and disjoint from M. Then there is a neighbourhood of $(x,v)$ in $SU$ such that for every $(x',v')$ in this neighbourhood the $\lambda$-geodesic $\gamma_{x',v'}$ will hit $\Phi$ at the time $\tilde\tau(x',v')$ smoothly depending on $(x',v')$ and such that $\tilde \tau(x,v)=\tau(x,v)$. For these $(x',v')$ we can therefore take $\tilde \tau(x',v')$ as $\tau(x',v')$ while defining $u$ by \eqref{def-chi}. So locally the lower limit of integration in \eqref{def-chi} can be chosen to be a smooth function. Since $\psi$ is continuous and lies in the Sobolev space $H^1(SU)$, this observation allows us to prove routinely that $u$ is continuous and belongs to $H^1(SU)$.

If $(x,v)\in RM$, then there is a neighbourhood of $(x,v)$ in $SU$ such that for all $(x',v')$ in this
neighbourhood the $\lambda$-geodesics $\gamma_{x',v'}$ intersect $\p M$ transversally from the same side as $\gamma_{x,v}$ so that
the interior of the segment between $x'$ and the point of intersection lies in $M^\text{int}$ for each of these $\lambda$-geodesics. Moreover, the parameter values $\tau(x',v')$ of the intersection
points are smooth functions of $(x',v')$ in this neighbourhood. Since $\psi$ is smooth in $SM$, we conclude that $u$ is smooth in the chosen neighbourhood of $(x,v)$ and therefore smooth on $RM$.

Finally, we give the proof of (iii). For $\varphi\in C^\infty (SM)$ and $(x,v)\in SM$ we have
\begin{equation}\label{def-F}
\mathbf F \varphi(x,v)=\frac{d}{dt}\varphi(\phi_t(x,v))\bigg|_{t=0}.
\end{equation}
By way of approximation, it is easy to see that \eqref{def-F} works equally well for the
functions $\varphi\in H^1(SU)\cap C(SU)$. To apply it to our function $u$, take $(x,v)\in SM$. Then $\gamma_{\phi_s(x,v)}(t)=\gamma_{x,v}(t+s)$ and we can
take $\tau(\phi_s(x,v))=\tau(x,v)-s$. By \eqref{def-chi}
$$
u(\phi_s(x,v))=\int^{\tau(\phi_s(x,v))}_0\psi(\phi_{t+s}(x,v))\,dt=\int^{\tau(x,v)}_s \psi(\phi_t(x,v))\,dt.
$$
Taking the derivative at $s = 0$ gives \eqref{kinetic-equation} and finishes the proof of the lemma.
\end{proof}

\subsection{Pestov integral identity}\label{3.3}
Let $D$ be a compact oriented two-dimensional submanifold of $U$ with boundary $\p D$ and $u: SD\to {\mathbb R}$ be a smooth function such that $u|_{\partial(SD)}=0$. By Theorem \ref{1st-int-id} we have
\begin{equation}\label{1st-int-id-b}
\int_{SD}({\bf F}Vu)^2\,d\mu-\int_{SD}\mathbb{K}(Vu)^2\,d\mu=\int_{SD}(V{\bf F}u)^2\,d\mu-\int_{SD}({\bf F}u)^2\,d\mu.
\end{equation}

\begin{Lemma}\label{1st-int-id-ns}
Let $D\subset U$ be a surface with boundary $\partial D$. Let a function $u: SD\to {\mathbb R}$ be such that $u\in H^1(SD)$, ${\bf F}u\in H^1(SD)$, $u$ is smooth in some neighbourhood of $\partial(SD)$ in $SD$, and $u|_{\partial(SD)}=0$. Then the integral identity (\ref{1st-int-id-b}) is valid for $u$.
\end{Lemma}
\begin{proof}
We want to construct a sequence of smooth functions $u_k$ coinciding with $u$ in a neighbourhood of $\partial(SD)$ and such that $u_k\to u$ in $H^1(SD)$ and ${\bf F}u_k\to {\bf F}u$ in $H^1(SD)$ as $k\to\infty$. Then applying (\ref{1st-int-id}) to each function $u_k$ and passing to the limit as $k\to\infty$, we come to the desired conclusion.

A chart $(O,\varphi)$ of the manifold $SN$ is called {\it straightening} the vector field $\mathbf F$ if $\psi_*\mathbf F$ coincides with a coordinate vector field on the range $\psi(O)\subset \mathbb R^{2}$.
Cover $SD$ by an atlas consisting of sufficiently small straightening charts $(O_i,\varphi_i)$, so that the charts having nonempty intersection with $\p(SD)$ lie entirely in that neighbourhood of $\p(SD)$ where $u$ is smooth. Choose a partition of unity $\{\mu_i\}$ subordinate to this atlas. Then $u=\sum_i\mu_i u$, where each $\mu_i u$ has support in $O_i$.

Choose those indices $i$ for which $O_i$ is disjoint from $\p(SD)$.
Fix a nonnegative function $\kappa\in C^{\infty}_0(\mathbb R^{2})$ such that
$\int_{\mathbb R^{2}}\kappa\,dx=1$. Set $\kappa_{\delta}(x)=\kappa(x/{\delta})/{\delta^{2}}$
for $\delta>0$. The function $\tilde{u}^{\delta}_i=((\mu_iu)\circ\varphi_i^{-1})*\kappa_{\delta}$
is $C^{\infty}$-smooth on $\mathbb R^{2}$ and $\supp\tilde{u}^{\delta}_i\subset \varphi_i(O_i)$
for $\delta>0$ small enough. The difference $\tilde{u}^{\delta}_i-(\mu_iu)\circ\varphi_i^{-1}$
tends to zero uniformly on $\mathbb R^{2}$ as $\delta\to0$.
Lift the function $\tilde{u}^{\delta}_i$ to $O_i$ by $\varphi_i$,
i.e. consider the function $\tilde{u}^{\delta}_i\circ\varphi_i$
and denote it again by $\tilde{u}^{\delta}_i$.
For the other indices, take the functions $\mu_iu$ themselves for $\tilde{u}^{\delta}_i$.
Then the sum $u_{\delta}=\sum_i\tilde{u}^{\delta}_i$
is the sought approximation to $u$.
\end{proof}

\subsection{Riccati equation and second integral identity}\label{3.4}
For $(x,v)\in SM$, define
$$
{\mathcal V}(x,v):=\ker d_{(x,v)}\pi,\quad\text{and}\quad E(x,v):={\mathcal V}(x,v)\oplus\mathbb{R} {\bf F}(x,v),
$$
where $d_{(x,v)}\pi$ is the differential of the natural projection $\pi:SM\to M$ and $\mathbf F$
is the infinitesimal generator of the thermostat flow.

\begin{Lemma}\label{r-exists}If  $\gamma:[0,T]\to M$ is a $\lambda$-geodesic, then
$$
d_{\dot{\gamma}(0)}\phi_t(E)\cap {\mathcal V}(\gamma(t),\dot{\gamma}(t))=\{0\}
$$
for every $t\in(0,T]$.
\end{Lemma}

\begin{proof}
Take $(x,v)\in SM$ and $t\in(0,T]$. From the definition of $\exp^{\lambda}$ it is straightforward that
$$\text{image}(d_{tv}\exp^{\lambda}_x)=d_{\dot{\gamma}(t)}\pi(d_{\dot{\gamma}(0)}\phi_t(E)).$$
By the absence of conjugate points, $d_{w}\exp^{\lambda}_x$ is a linear isomorphism for every $w\in T_x M$ at which $\exp^{\lambda}_x$ is defined, and the lemma follows.
\end{proof}

For $(x,v)\in SM$ there is $t_0$ such that $x_0=\gamma_{x,v}(t_0)\in U\setminus M$.
If $x_0$ is close enough to $M$, then $\gamma_{x_0,v_0}$ has no conjugate points either,
where $v_0=\dot{\gamma}_{x,v}(t_0)$. Then Lemma \ref{r-exists} implies that there is
a unique continuous function $r(t)=r(\phi_t(x_0,v_0))$ along the orbit of $(x_0,v_0)$ such that
$$
H(\phi_t(x_0,v_0))+r(t)V(\phi_t(x_0,v_0))\in d_{v_0}\phi_t(E).
$$
In this way we can define $r$ on the whole $SM$.
The so-obtained function $r$ is smooth on every orbit and $r\in L^{\infty}(SM)$.
Below we will need to use the fact that the function $r$ satisfies a Riccati type equation along the flow.

\begin{Lemma}\label{riccati}{\bf{(Riccati equation). }}The function r satisfies
$${\bf F}(r-V(\lambda))+r(\lambda I-V(\lambda)+r)+\mathbb{K}-\lambda IV(\lambda)=0.$$
\end{Lemma}

\begin{proof}Fix $(x,v)\in SM$ and set
$$
\xi(t):=d\phi_{-t}(H(\phi_t(x,v))+r(\phi_t(x,v))V(\phi_t(x,v))).
$$
By the definition of $r$, $\xi(t)\in E(x,v)$ for all $t$. Differentiating with respect to $t$ and setting $t=0$ we obtain:
$$\dot{\xi}(0)=[{\bf F},H]+{\bf F}(r)V+r[{\bf F},V].$$
Using the commutation relations \eqref{cmr2} we have
$$\dot{\xi}(0)=-\lambda {\bf F}-\lambda I\xi(0)+\{K-H(\lambda)-\lambda J+\lambda^2+{\bf F}(r)-rV(\lambda)\}V.$$
Replacing $H$ by $\xi(0)-rV$ yields:
$$\dot{\xi}(0)+(r+\lambda I)\xi(0)-\lambda {\bf F}=\{K-H(\lambda)-\lambda J+\lambda^2+{\bf F}(r)+\lambda Ir-rV(\lambda)+r^2\}V.$$
Since $\dot{\xi}(0)+(r+\lambda I)\xi(0)-\lambda {\bf F}\in E$ we must have
$$K-H(\lambda)-\lambda J+\lambda^2+{\bf F}(r)+\lambda Ir-rV(\lambda)+r^2=0$$
which is the desired equation.
\end{proof}

For the proof of Theorem \ref{theorem A} we also need the following result.

\begin{Theorem}\label{2nd-int-id} Let $\varphi\to\mathbb R$
be a function vanishing on $\partial(SM)$ such that $\varphi\in C^{\infty}(RM)$. Then
\[\int_{SM}({\bf F}\varphi)^2\,d\mu-\int_{SM}\mathbb K \varphi^2\,d\mu=
\int_{SM}[{\bf F}(\varphi)-r\varphi+\varphi V(\lambda)]^2\,d\mu\geq 0.\]
Moreover,
\[\int_{SM}[{\bf F}(\varphi)-r\varphi+\varphi V(\lambda)]^2\,d\mu=0\]
if and only if $\varphi=0$ on $RM$.
\end{Theorem}

\begin{proof} Let us expand $[{\bf F}(\varphi)-r\varphi+\varphi V(\lambda)]^2$:
\begin{align*}
[{\bf F}(\varphi)-r\varphi+\varphi V(\lambda)]^2&=[{\bf F}(\varphi)]^2+\varphi^2 r^2+\varphi^2[V(\lambda)]^2\\
&-2{\bf F}(\varphi)\varphi r+2{\bf F}(\varphi)\varphi V(\lambda)-2\varphi^2 r V(\lambda).
\end{align*}
Using Lemma \ref{riccati}, we obtain:
\begin{align*}
[{\bf F}(\varphi)-r\varphi+\varphi V(\lambda)]^2&=[{\bf F}(\varphi)]^2-\mathbb K\varphi^2\\
&-{\bf F}((r-V(\lambda))\varphi^2)+\varphi^2[V(\lambda)]^2\\
&-\varphi^2 r[\lambda I+V(\lambda)]+\lambda IV(\lambda)\varphi^2.
\end{align*}
If we integrate the last equality with respect to the measure $\mu$,
we obtain as desired:
\[\int_{SM}({\bf F}\varphi)^2\,d\mu-\int_{SM}\mathbb K \varphi^2\,d\mu=
\int_{SM}[{\bf F}(\varphi)-r\varphi+\varphi V(\lambda)]^2\,d\mu,\]
since by the Stokes Theorem and (\ref{lie1}), we have
\begin{multline*}
\int_{SM}{\bf F}((r-V(\lambda))\varphi^2)\,d\mu
=\int_{SM}\{\varphi^2[V(\lambda)]^2-\varphi^2 r[\lambda I+V(\lambda)]+\lambda IV(\lambda)\varphi^2\}\,d\mu
\\
+\int_{\p(SM)}(r-V(\lambda))\varphi^2\,i_{\bf F}\Theta
\end{multline*}
while the last integral vanishes due to the boundary condition.

Suppose now
$$
\int_{SM}[{\bf F}(\varphi)-r\varphi+\varphi V(\lambda)]^2\,d\mu=0,
$$
which implies
\[{\bf F}(\varphi)-r\varphi+\varphi V(\lambda)=0\]
on $RM$. This means that, on almost every orbit, the function $\varphi$ satisfies a homogeneous first-order ordinary differential equation with zero boundary data. This surely implies that $\varphi\equiv 0$ on such an orbit, which yelds $\varphi\equiv 0$ on $RM$.
\end{proof}

\subsection{End of the proof of Theorem \ref{theorem A}}\label{3.5}
Let $W\subset N$ be a collar neighbourhood of $\partial M$ in $N$.
This means that there is a diffeomorphism $\Psi: \partial M \times (-1,1)\to W$
such that the restriction $\Psi|_{\partial M \times \{0\}}$ is the identity map. We also assume that $\Psi(\partial M \times (-1,0))\subset M$ and $\Psi(\partial M \times (0,1))\subset N\setminus M$.

Put $M_{\varepsilon}=M\cup \Psi(\partial M \times [0,\varepsilon])$, $0\le\varepsilon<1$, obtaining a nested family of subdomains in $N$, with $M_0=M$.

Let us now prove Theorem \ref{theorem A}. By Lemma \ref{chi-prop} the function $u$
satisfies the condition of Lemma \ref{1st-int-id-ns} with $D=M_{\varepsilon}$
for every $\varepsilon>0$. Then by Lemma \ref{1st-int-id-ns} we have
$$
\int_{SM_{\varepsilon}}({\bf F}Vu)^2\,d\mu-\int_{SM_{\varepsilon}}\mathbb{K}(Vu)^2\,d\mu=\int_{SM_{\varepsilon}}(V{\bf F}u)^2\,d\mu-\int_{SM_{\varepsilon}}({\bf F}u)^2\,d\mu
$$
for every $\varepsilon>0$. It is easy to see that the right-hand side of the last equation is nonpositive.
Indeed, since $V{\bf F}u=-V\alpha(x,v)$ we have
\[\int_{SM_\varepsilon}f\alpha\,d\mu=0\;\;\; \mbox{\rm and}\;\; \int_{SM_\varepsilon}\alpha^2\,d\mu=
\int_{SM_\varepsilon}(V\alpha)^2\,d\mu.\]
This follows from \cite[Lemma 4.4]{DPrigid}, which holds in any dimension. Thus
\[\int_{SM_{\varepsilon}}({\bf F}Vu)^2\,d\mu-\int_{SM_{\varepsilon}}\mathbb{K}(Vu)^2\,d\mu=
-\int_{SM_\varepsilon}f^{2}\,d\mu\leq 0.\]
Lemma \ref{chi-prop} allows us to pass to the limit as $\varepsilon\to 0$ in this identity. Using Theorem \ref{2nd-int-id} we obtain that $Vu\equiv 0$ on $RM$. This says that $u=-h$ almost everywhere for some $h\in C^\infty_0(M)$. Since $u\in C(SM)$, then $u=-h$ everywhere. But in this case, since
$d\pi_{(x,v)}({\bf F})=v$ we have
${\bf F}u=-dh_{x}(v)$. This clearly implies the claim of Theorem \ref{theorem A}.

\section{Proof of Theorem \ref{theorem C'}}
The considerations here generalize mostly  those in \cite{DPentr}, where the Riemannian case
was treated. We start Subsection \ref{4.1}
with derivation of the Riccati equation.
In Subsection \ref{4.2} we obtain from it a certain integral identity. The latter, combined with
the Pestov integral identity, will prove Theorem \ref{theorem C'} in Subsection \ref{4.3}.
Subsection \ref{4.4} plays an auxiliary role and is devoted to providing
a sufficient condition for a thermostat flow to be Anosov on using the
hyperbolicity test due to M.~Wojtkowski \cite{MW}.

\subsection{Riccati equation for Anosov thermostats}\label{4.1}
Recall that the Anosov property means that there is a continuous invariant splitting
$T(SM)=\mathbb{R} \mathbf F\oplus E^{u}\oplus E^{s}$ in such a way that
there are constants $C>0$ and $0<\rho<1<\eta$ such that
for all $t>0$ we have
\[\|d\phi_{-t}|_{E^{u}}\|\leq C\,\eta^{-t}\;\;\;\;\mbox{\rm
and}\;\;\;\|d\phi_{t}|_{E^{s}}\|\leq C\,\rho^{t},\]
where the norms are taken with respect to a Sasaki type Riemannian metric on $SM$
induced by the Finsler metric $F$. The subbundles are then invariant and H\"older continuous and have
smooth integral manifolds, the stable and unstable manifolds,
which define a continuous foliation with smooth leaves.

Let us introduce the weak stable and unstable bundles:
\[E^+=\mathbb{R} \mathbf F\oplus E^s,\]
\[E^-=\mathbb{R} \mathbf F\oplus E^u.\]

\begin{Lemma}\label{VnotinE}For any $(x,v)\in SM$, $V(x,v)\notin E^{\pm}(x,v)$.
\end{Lemma}
\begin{proof}
Let $\Lambda(SM)$ be the bundle above $SM$ that at each point $(x,v)\in SM$
consists of all 2-dimensional subspaces $W$ of $T_{(x,v)}SM$ with
$\mathbf F(x,v)\in W$.

The map $(x,v)\mapsto \R \mathbf F(x,v)\oplus \R V(x,v)$ is a section of
$\Lambda(SM)$ and its image is a codimension one submanifold that we denote by $\Lambda_{V}$ and call it {\it Maslov cycle}.
Similarly the map $(x,v)\mapsto \R \mathbf F(x,v)\oplus \R H(x,v)$ is a section of
$\Lambda(SM)$ and its image is a codimension one submanifold that we denote by $\Lambda_{H}$.

The flow $\phi$ naturally lifts to a flow $\phi^*$ acting on $\Lambda(SM)$ via its differential.
Let $\mathbf F^*$ be the infinitesimal generator of $\phi^*$.

\begin{Claim}\label{claim}
$\mathbf F^*$ is transversal to the Maslov cycle $\Lambda_V$.
\end{Claim}
\begin{proof}
Indeed, define a function $m:\Lambda(SM)\setminus \Lambda_{H}\to\R$
as follows. If $W\in\Lambda(SM)\setminus \Lambda_{H}$, then $H\notin W$.
Thus there exists a unique $m=m(W)$ such that $mH+V\in W$. Clearly $m$ is smooth
and $\Lambda_{V}=m^{-1}(0)\subset\Lambda(SM)\setminus \Lambda_{H}$.
Fix $(x,v)\in SM$ and
set
$$
m(t):=m(\phi^*_{t}(\R F(x,v)\oplus \R V(x,v))).
$$
By the definition of $m$, there exist functions
$x(t)$ and $y(t)$ such that

\[m(t)H(\phi_t(x,v))+V(\phi_t(x,v))=x(t)\mathbf F(\phi_t(x,v))+y(t)d\phi_{t}(V(x,v)).\]
Equivalently
\[m(t)d\phi_{-t}(H(\phi_t(x,v)))+d\phi_{-t}(V(\phi_t(x,v)))=x(t)\mathbf F(x,v)+y(t)V(x,v).\]
Differentiating with respect to $t$ and setting $t=0$ (recall that $m(0)=0$)
we obtain:
\[\dot{m}(0)H+[\mathbf F,V]=\dot{x}(0)\mathbf F+\dot{y}(0)V.\]
But $[V,\mathbf F]=H+V(\lambda)V$. Thus $\dot{m}(0)=1$ which proves the claim.
\end{proof}

From the Claim \ref{claim} it follows that $\Lambda_V$ determines an oriented
codimension one cycle in $\Lambda(SM)$ and by duality it defines
a cohomology class ${\mathfrak m}\in H^{1}(\Lambda(SM),\Z)$. Set $E=E^\pm$.
Given a continuous closed curve $\alpha:S^1\to SM$, the {\it index}
of $\alpha$ is $\nu(\alpha):=\langle {\mathfrak m}, [E\circ\alpha]\rangle$
(i.e. $\nu=E^*{\mathfrak m}\in H^{1}(SM,\Z)$).
The index of $\alpha$ only depends on the homology class
of $\alpha$.
Since $E$ is $\phi$-invariant, the Claim \ref{claim} also ensures that if
$\gamma$ is any closed orbit of $\phi$, then $\nu(\gamma)\geq 0$.

Recall that according to Ghys \cite{Ghy} we know that $\phi$ is topologically
conjugate to the geodesic flow of a metric of constant negative curvature.
In particular, every homology class in $H_{1}(SM,\Z)$ contains
a closed orbit of $\phi$. Thus $\nu$ must vanish.

If there exists $(x,v)\in SM$ for which
$V(x,v)\in E(x,v)$, then using that every point of $\phi$ is non-wandering,
we can produce exactly as in \cite[Lemma~2.49]{P1} a closed curve
$\alpha:S^1\to SM$ with $\nu(\alpha)>0$.
This contradiction shows the lemma.
\end{proof}

Lemma \ref{VnotinE} implies that there exist unique continuous functions $r^\pm(x,v)$ on $SM$ such that
\begin{align*}
H(x,v)+r^+(x,v)V(x,v)\in E^+,\\
H(x,v)+r^-(x,v)V(x,v)\in E^-.
\end{align*}
Note that the Anosov property implies that $r^\pm$ are smooth along $\phi$ and $r^+\ne r^-$ everywhere. Next, we show that the functions $r^\pm$ satisfy a Riccati type equation along the flow.

\begin{Lemma}The function $r=r^\pm$ satisfies
\begin{equation}\label{riccati-eq}
{\bf F}(r-V(\lambda))+r(\lambda I-V(\lambda)+r)+\mathbb{K}-\lambda IV(\lambda)=0.
\end{equation}
\end{Lemma}

\begin{proof}Let $E=E^\pm$. Fix $(x,v)\in SM$ and set
$$
\xi(t):=d\phi_{-t}(H(\phi_t(x,v))+r(\phi_t(x,v))V(\phi_t(x,v))).
$$
By the definition of $r$, $\xi(t)\in E(x,v)$ for all $t$. Differentiating with respect to $t$ and setting $t=0$ we obtain:
$$\dot{\xi}(0)=[{\bf F},H]+{\bf F}(r)V+r[{\bf F},V].$$
Using the commutation relations \eqref{cmr2} we have
$$\dot{\xi}(0)=-\lambda {\bf F}-\lambda I\xi(0)+\{K-H(\lambda)-\lambda J+\lambda^2+{\bf F}(r)-rV(\lambda)\}V.$$
Replacing $H$ by $\xi(0)-rV$ yields:
$$\dot{\xi}(0)+(r+\lambda I)\xi(0)-\lambda {\bf F}=\{K-H(\lambda)-\lambda J+\lambda^2+{\bf F}(r)+\lambda Ir-rV(\lambda)+r^2\}V.$$
Since $\dot{\xi}(0)+(r+\lambda I)\xi(0)-\lambda {\bf F}\in E$ we must have
$$K-H(\lambda)-\lambda J+\lambda^2+{\bf F}(r)+\lambda Ir-rV(\lambda)+r^2=0$$
which is the desired equation.
\end{proof}

\subsection{Second integral identity for Anosov thermostats}\label{4.2}
Here we prove the following integral identity which we use in the proof of Theorem \ref{theorem C}.
\begin{Theorem}\label{2nd-int-id-Anosov}Let $\varphi:SM\to \mathbb{R}$ be
a smooth function and suppose the flow $\phi_t$ is Anosov. Then for $r=r^\pm$
\[\int_{SM}({\bf F}\varphi)^2\,d\mu-\int_{SM}\mathbb{K}\varphi^2\,d\mu=\int_{SM}[{\bf F}(\varphi)-r\varphi+\varphi V(\lambda)]^2\,d\mu\geq 0.\]
Moreover,
\[\int_{SM}[{\bf F}(\varphi)-r\varphi+\varphi V(\lambda)]^2\,d\mu=0\]
if and only if $\varphi=0$.
\label{mainth}
\end{Theorem}

\begin{proof} We omit the proof of the first part since it is exactly the same as the first part of Theorem \ref{2nd-int-id}.
Suppose now
\[\int_{SM}[{\bf F}(\varphi)-r\varphi+\varphi V(\lambda)]^2\,d\mu=0,\]
which implies
\[{\bf F}(\varphi)-r\varphi+\varphi V(\lambda)=0\]
everywhere. Since this holds for $r=r^\pm$, we deduce:
$$(r^+-r^-)\psi=0.$$
But for an Anosov flow $r^+-r^-\ne0$. This surely implies that $\varphi\equiv 0$ on $SM$.
\end{proof}

\subsection{End of the proof of Theorem \ref{theorem C'}}\label{4.3}
Assume we have a cohomological equation
$$
\mathbf Fu(x,v)=f(x)+\alpha_x(v).
$$
By the Pestov integral identity we have
$$
\int_{SM}({\bf F}Vu)^2\,d\mu-\int_{SM}\mathbb{K}(Vu)^2\,d\mu=\int_{SM}(V{\bf F}u)^2\,d\mu-\int_{SM}({\bf F}u)^2\,d\mu.
$$
It is easy to see that the right-hand side of the last equation is nonpositive.
Indeed, since $V{\bf F}u=V\alpha(x,v)$ we have
\[\int_{SM}f\alpha\,d\mu=0\;\;\; \mbox{\rm and}\;\; \int_{SM}\alpha^2\,d\mu=
\int_{SM}(V\alpha)^2\,d\mu.\]
This follows from \cite[Lemma 4.4]{DPrigid}, which holds in any dimension. Thus
\[\int_{SM}({\bf F}Vu)^2\,d\mu-\int_{SM}\mathbb{K}(Vu)^2\,d\mu=
-\int_{SM}f^{2}\,d\mu\leq 0.\]
Using Theorem \ref{2nd-int-id-Anosov} we obtain that $Vu\equiv 0$ on $SM$,
which says that $u=h$ some $h\in C^\infty(M)$. Since $d\pi_{(x,v)}({\bf F})=v$ we have ${\bf F}u=dh_{x}(v)$.
This clearly implies the claim of Theorem \ref{theorem C'}.

\subsection{Sufficient condition for a thermostat flow  to be Anosov}\label{4.4}

\begin{Theorem}\label{theorem D}
If $(M,F,\lambda)$ is a thermostat on a closed oriented Finsler surface and
$$
K-H(\lambda)-\lambda J+\lambda^2+\frac{(\lambda I+V(\lambda))^2}{4}<0,
$$
then the flow $\phi_t$ is Anosov.
\end{Theorem}

\begin{proof}
We apply the hyperbolicity test  of \cite[Theorem 5.2]{MW}.

Given $\xi\in T_{(x,v)}(SM)$ we may write:
$$\xi=a{\bf F}(x,v)+yH(x,v)+zV(x,v).$$
Define a quadratic form ${\mathbb Q}$ on $SM$ by
$${\mathbb Q}(\xi)=yz.$$
Consider the {\it quotient bundle} $\hat{T}(SM)$ defined by
$$\hat{T}_{(x,v)}(SM):=T_{(x,v)}(SM)/{\mathbb R}{\bf F}(x,v).$$
Since $d\phi_t{\bf F}(x,v)={\bf F}(\phi_t(x,v))$, the differential $d\phi_t$
descends to the quotient to define a map $A_t: \hat{T}_{(x,v)}(SM)\to \hat{T}_{\phi_t(x,v)}(SM)$ satisfying
$$A_{s+t}=A_s\circ A_t.$$

By Theorem 5.2 of \cite{MW}, it suffices to prove that the flow $\phi_t$ is strictly monotone
with respect to the quadratic form $\mathbb Q$, i.e., the projection of the Lie derivative
$\mathcal L_\mathbf F \mathbb Q$ onto the quotient bundle is positive definite. To this end,
we need the following lemma.

\begin{Lemma}\label{Jacobi-system}
For $\xi\in T_{(x,v)}SM$ consider the representation of $d\phi_t(\xi)$ in terms of $\mathbf F, H, V$:
$$
d\phi_t(\xi)=a(t)\mathbf F(\phi_t(x,v))+y(t)H(\phi_t(x,v))+z(t)V(\phi_t(x,v)).
$$
Then the functions $a,y,z$ satisfy the following equations
\begin{align*}
\dot{a}&=\lambda y,\\
\dot{y}&=\lambda I y+z,\\
\dot{z}&=-\{K-H(\lambda)-\lambda J+\lambda^2\}y+V(\lambda)z.
\end{align*}
\end{Lemma}

\begin{proof}
Equivalently we have
$$
\xi=a(t)d\phi_{-t}({\bf F}(\phi_t(x,v)))+y(t)d\phi_{-t}(H(\phi_t(x,v)))+z(t)d\phi_{-t}(V(\phi_t(x,v))).
$$
If we differentiate the last equality with respect to $t$ we obtain:
$$
0=\dot{a}{\bf F}+\dot{y}H+y[{\bf F},H]+\dot{z}V+z[{\bf F},V].
$$
Using the commutation relations \eqref{cmr2} and regrouping, we find out the equations of Lemma \ref{Jacobi-system}.
\end{proof}

Continuing the prof of the theorem, we have
\begin{align*}
\mathcal L_\mathbf F \mathbb Q|_{\hat{T}(SM)}&= \frac{d}{dt}\Big|_{t=0}\hat{\mathbb Q}(A_t(\xi))=\frac{d}{dt}\Big|_{t=0}(y(t)z(t))\\
&=-\{K-H(\lambda)-\lambda J+\lambda^2\}y^2+(V(\lambda)+\lambda I)yz+z^2.
\end{align*}
By Sylvester's criterion it is positively definite if and only if
$$K-H(\lambda)-\lambda J+\lambda^2+\frac{(\lambda I+V(\lambda))^2}{4}<0,$$
which concludes the proof of the theorem.
\end{proof}

\section{Proof of Theorem \ref{theorem E}}
In Subsection \ref{5.1} we give an equivalent condition to the absence of
conjugate points for a thermostat in terms of the Jacobi equation. This is completely
analogous to the case of the geodesic flow.
Subsection~\ref{5.2} is devoted to constructing an integrable solution of
the corresponding Riccati equation.
This is achieved by reducing the problem to Hopf's construction in \cite{H}
and is similar to the considerations we used before in \cite[Sections 4-5]{AD}.
Finally, in Subsection~\ref{5.3} we conclude the proof of Theorem~\ref{theorem E}
by applying the Pestov integral identity.

\subsection{Thermostats without conjugate points}\label{5.1}
Put
$$
q=-\lambda I-V(\lambda),\quad k=\mathbb K-\mathbf FV(\lambda)-{\bf F}(\lambda I).
$$

Our aim in this subsection is the following

\begin{Theorem}\label{MSAPP}
A thermostat $(M,F,\lambda)$ has no conjugate points if and only if
all solutions of the Jacobi equation \begin{equation}\label{Jacobi-eq}
\ddot{y}+q\dot{y}+ky=0,
\end{equation}
along any $\lambda$-geodesic, vanish at most once.
\end{Theorem}

We consider a variation of the $\lambda$-geodesic $\gamma(t)=\pi\circ\phi_t(x,v)$
for some $(x,v)\in SM$. We take the variation $c(s,t)=\pi(\phi_t(Z(s)))$ of $\gamma$
that depends on a curve $Z\subset TM$ with $\dot Z(0)=\xi\in T_{(x,v)}SM$.
The vector field $J_{\xi}(t):=\de{c}{s}\Big|_{s=0}(t)$ (that depends on $\xi$) is called
a {\it Jacobi field} along $\gamma$.

\begin{Lemma}\label{cor1}Every Jacobi field $J_{\xi}$, expressed as
$$
J_\xi(t)=x(t)\dot\gamma(t)+y(t)i\dot\gamma(t),
$$
satisfies the {\it Jacobi equations}
\begin{align}
\dot{x}&=\lambda y,\label{Jacobi-eqx}\\
\ddot{y}&+q\dot{y}+ky=0.\label{Jacobi-eqy}
\end{align}
In particular, if a Jacobi field $J$ is tangent to the $\lambda$-geodesic $\gamma$ everywhere,
then $J = c\gamma$, where $c=const$.
\end{Lemma}
\begin{proof}
For $\xi\in T(SM)$, write
$$d\phi_t(\xi)=x(t){\bf F}+y(t)H+z(t)V,$$
equivalently
$$\xi=x(t)d\phi_{-t}({\bf F})+y(t)d\phi_{-t}(H)+z(t)d\phi_{-t}(V).$$
If we differentiate the last equality with respect to $t$, we obtain:
$$0=\dot{x}{\bf F}+\dot{y}H+y[{\bf F},H]+\dot{z}V+z[{\bf F},V].$$
Using the bracket relations and regrouping we arrive at \eqref{Jacobi-eqx} and \eqref{Jacobi-eqy}.
\end{proof}

Let $\gamma:[0,T]\to M$ be a $\lambda$-geodesic with endpoints $x=\gamma(0)$ and $y=\gamma(T)$.
We say that $x$ and $y$ are {\it conjugate along $\gamma$} if
the map $d_{T\dot\gamma(0)}\exp_x^{\lambda}$ has a non-maximal rank.
Note that this definition does not contradict the definition of the absence of
conjugate points mentioned in the Introduction.

There exists a simple but very useful and well-known relation
between the exponential map and the Jacobi fields:

\begin{Theorem}\label{conj-along-geodesic}
Let $\gamma:[0,T]\to M$ be a $\lambda$-geodesic with endpoints $x=\gamma(0)$ and $y=\gamma(T)$.
Then $x$ and $y$ are conjugate along $\gamma$ if and only if there exists a nonzero
Jacobi field along $\gamma$ satisfying $J(0)=J(T)=0$.
\end{Theorem}
\begin{proof}
For the proof we need the following

\begin{Lemma}\label{0,w}
Let $\gamma:[0,T]\to M$ be a $\lambda$-geodesic such that
$\gamma(t)=\exp^{\lambda}_x(tv)$, with $x=\gamma(0)$ and $v=\dot\gamma(0)$.
Let $w\in T_x M$. Then $J(t)=d_{tv}\exp^{\lambda}_x(tw)$ is a Jacobi field along $\gamma$.
Moreover, $J(0)=0$ and $D_{t}J(0)=w$.
\end{Lemma}
The proof of the lemma is standard and we give it for completeness
after the proof of theorem.

Suppose there exists a nonzero vector $w\in T_x M$ such that $d_{v}\exp^{\lambda}_x(w)=0$.
By Lemma~\ref{0,w}, the Jacobi field $J(t)=d_{tv}\exp^{\lambda}_x(tT^{-1}w)$ is nontrivial,
since $D_{t}J(0)=T^{-1}w\neq0$, and it satisfies $J(0)=J(T)=0$.

Conversely, if the points $x$ and $y$ are conjugate along $\gamma$
then there exists a nontrivial Jacobi field $J$ along $\gamma$ such that $J(0)=J(T)=0$.
Let $D_t J(0)=w$. Then $w\neq0$ and by Lemma \ref{0,w} $d_{T\dot\gamma(0)}\exp^{\lambda}_x(Tw)=J(T)=0$ so that $y=\exp^{\lambda}_x(T\dot\gamma(0))$.
This means that $w\in\ker d_{T\dot\gamma(0)}\exp^{\lambda}_x$.
\end{proof}

\begin{proof}[Proof of Lemma \ref{0,w}]
Consider the variation $c(s,t)=\exp^{\lambda}_x(t(v+sw))$ of $\gamma$. Since
$$
\de{c}{s}(s,t)=\de{\exp^{\lambda}_x(t(v+sw))}{s}=d_{t(v+sw)}\exp^{\lambda}_x(tw),
$$
the vector field $J(t)$ is a Jacobi field. Once $d_0\exp^{\lambda}_x$ is the identity map,
we have
$$
J(0)=d_{0}\exp^{\lambda}_x(0)=0.
$$
It is well known that $D_s Y(s,t)=D_t J(s,t)$, where $Y(s,t)=\de{c}{t}(s,t)$ and $J(s,t)=\de{c}{s}(s,t)$. Then
$$
D_t J(s,0)=D_sY(s,0)=D_s(d_0\exp^{\lambda}_x(v+sw))=\de{(v+sw)}{s}=w,
$$
which finishes the proof.
\end{proof}

Let $\xi\in T_{(x,v)}TM$, and let $Z(t)=(\alpha(t),z(t))$ be any curve with $Z(0)=(x,v)$ and $\dot{Z}(0)=\xi$. Define the {\it connection map}
$$
{\mathcal K}_{(x,v)}(\xi):=\nabla_{\dot{\alpha}}\dot{z}(0).
$$
For $(x,v)\in TM$, define the horizontal subbundle by
$\mathcal H(x,v)=\ker\mathcal K_{(x,v)}$. So we obtain the following isomorphism:
$$
T_{(x,v)}TM\to \mathcal H(x,v)\oplus \mathcal V(x,v),\quad
\xi\mapsto(d\pi_{(x,v)}(\xi),\mathcal K_{(x,v)}(\xi)).
$$

\begin{proof}[Proof of Theorem~\ref{MSAPP}]
Assume that a thermostat has no conjugate points.
Let $\gamma(t)$, $0\le t\le T$, be a $\lambda$-geodesic.
By Lemma~\ref{r-exists}, $d_{\dot{\gamma}(0)}\phi_t(E)$ can be seen to be
a graph over the horizontal subspace for $t\in(0,T]$. We may thus express
$$
d_{\dot{\gamma}(0)}\phi_t(E)=\graph S:=\{(v,S(t)v), v\in \mathcal H(\dot\gamma(t)\}
$$
with some $S(t):\mathcal H(\dot\gamma(t))\to \mathcal V(\dot\gamma(t))$ for $t\in(0,T]$.
We have (see, for instance, \cite[Lemma 3.1]{JP})
\begin{equation}\label{dp=j,dj}
d\phi_t(\xi)=(J_{\xi}(t),\dot J_{\xi}(t)).
\end{equation}
Let $u(t):=\langle S(t)i\dot\gamma,i\dot\gamma\rangle$.
In view of \eqref{dp=j,dj}, $\dot J_{\eta}=SJ_{\eta}$
for all $\eta\in T_{(x,v)}SM$. Using all of this,
we obtain that $\dot y=uy$. Since $u$ is well defined for all $t\in(0,T]$,
it is easy to see that $y$ never vanishes for $t\in(0,T]$.

Conversely, suppose that, for any $\lambda$-geodesic $\gamma(t)$,
any solution of \eqref{Jacobi-eq} on $\gamma$ with $y(a)=y(b)=0$
for some $a<b$ is necessarily $y\equiv0$.
Let $J(t)$ be a Jacobi field along $\gamma$ with $J(a)=J(b)=0$.
Using Lemma \ref{cor1} we then conclude that $J(t)$ is of the form
$J=c\dot\gamma$. As soon as $J(a)=J(b)=0$, $J$ vanishes identically.
Appealing to Theorem \ref{conj-along-geodesic}, we conclude the proof.
\end{proof}

\subsection{Riccati equation}\label{5.2}
Let $\gamma(t)$, $-\infty<t<+\infty$, be a complete unit speed $\lambda$-geodesic.
The Jacobi equation on $\gamma$ is:
\begin{equation}\label{j-e}
\ddot{y}+q\dot{y}+ky=0.
\end{equation}
If $y(t)$ is a nowhere vanishing solution of \eqref{j-e},  then
$r(t)=\frac{\dot y(t)}{y(t)}$ is a solution of the {\it Riccati equation}
\begin{equation}\label{ric-1}
\dot{r}+r^2+qr+k=0.
\end{equation}

Let
$$
m(t):=\exp{\left(-\frac12\int q(t)\,dt\right)}.
$$
If $y(t)=m(t)z(t)$ then $z(t)$ is a solution of
the equation
\begin{equation}\label{j-e-z}
\ddot{z}+\tilde k z=0,
\end{equation}
where
\begin{equation}\label{tilde-k}
\tilde k(t)=k(t)-\frac{\dot q}2+\frac{q^2}4.
\end{equation}
Since $m(t)$ is nowhere zero,  equation \eqref{j-e} has no conjugate points if
and only if so does equation \eqref{j-e-z}.

The Riccati equation corresponding to \eqref{j-e-z} is
\begin{equation}\label{ric-2}
\dot{u}+u^2+\tilde k=0.
\end{equation}
Clearly, the solutions of \eqref{ric-1} and \eqref{ric-2} are related by
\begin{equation}\label{u-r}
r(t)=u(t)-q(t)/2.
\end{equation}

Observe that,  once $SM$ is compact, there is a constant $A\ge0$ such that
$$
|\tilde k(x,v)|=\left|k(x,v)-\frac{F(q(x,v))}2+\frac{q^2(x,v)}4\right|\le A^2
$$
for all $(x,v)\in SM$. Since $\tilde k(t)$ is the restriction of
$\tilde k(x,v)$ to $(\gamma,\dot\gamma)$, we have
$$
|\tilde k(t)|\le A^2.
$$

In \cite{H},  Hopf constructed a solution $u(t)$ of \eqref{ric-2} such that $|u(t)|\le A$
for all $t$. Considering all $\gamma$ gives a bounded function $u(x,v)$ on $SM$ whose
resctriction to any $\gamma$ is a solution of \eqref{ric-2}, and Hopf proves in \cite{H}
that this $u(x,v)$ is a measurable function on $SM$.
In view of \eqref{u-r},  taking $r(x,v)=u(x,v)-q(x,v)/2$
then yields a bounded measurable function $r(x,v)$ whose restriction to any $\gamma$
is a solution of \eqref{ric-1}. From \eqref{ric-1} we readily infer that $r(x,v)$
satisfies the following equation on $SM$:
\begin{equation}\label{riccati-1}
\mathbf F(r)+r^2+qr+k=0.
\end{equation}

\subsection{End of the proof of Theorem \ref{theorem E}}\label{5.3}
Using the same arguments as in the proof of Theorem~\ref{1st-int-id},
we can derive the following integral identity for any $\varphi\in C^\infty(SM)$:
\begin{equation}\label{2nd}
\int_{SM}({\bf F}\varphi)^2\,d\mu-\int_{SM}\mathbb K \varphi^2\,d\mu=
\int_{SM}[{\bf F}(\varphi)-r\varphi+\varphi(\lambda I+V(\lambda))]^2\,d\mu\geq 0.
\end{equation}
Applying Theorem \ref{1st-int-id} for a smooth function $u:SM\to\mathbb R$, we get
\begin{equation}\label{c-1-int-id}
\int_{SM}({\bf F}Vu)^2\,d\mu-\int_{SM}\mathbb{K}(Vu)^2\,d\mu=\int_{SM}(V{\bf F}u)^2\,d\mu-\int_{SM}({\bf F}u)^2\,d\mu,
\end{equation}

If ${\bf F}u=f$, then it is obvious that
the right-hand side of \eqref{c-1-int-id} is nonpositive. Thus
$$
\int_{SM}(V{\bf F}u)^2\,d\mu-\int_{SM}({\bf F}u)^2\,d\mu=
-\int_{SM}f^2\,d\mu\leq 0.
$$
Combining this with \eqref{2nd} written for $\varphi=Vu$ yields $f=0$.
The proof of Theorem~\ref{theorem E} is complete.

\end{document}